\documentclass[a4paper,10pt]{amsart}

\usepackage[utf8]{inputenc}

\usepackage{graphicx}

\usepackage{cite}

\usepackage{hyperref}

\hypersetup{%
pdftitle={},
pdfsubject={Mathematics},
pdfauthor={Manfred Madritsch},
pdfkeywords={}
hyperindex=true,plainpages=false}

\usepackage{a4wide}

\usepackage{amsmath}
\usepackage{amsfonts}
\usepackage{amssymb}
\usepackage{amsthm}
\usepackage[initials]{amsrefs}

\AtBeginDocument{%
   \def\MR#1{}
}

\usepackage{xcolor}

\newtheorem{lem}{Lemma}[section]
\newtheorem{thm}[lem]{Theorem}
\newtheorem{prop}[lem]{Proposition}

\numberwithin{equation}{section}

\newtheorem*{cor*}{Corollary}
\newtheorem*{thm*}{Theorem}

\theoremstyle{definition}

\theoremstyle{remark}

\allowdisplaybreaks[3]

\newcommand{\N}{\mathbb{N}}
\newcommand{\Z}{\mathbb{Z}}

\newcommand{\R}{\mathbb{R}}

\renewcommand{\lvert}{\left\vert}
\renewcommand{\rvert}{\right\vert}
\renewcommand{\lVert}{\left\Vert}
\renewcommand{\rVert}{\right\Vert}

\newcommand{\entropy}{\mathrm{h}_{\mathrm{Sh}}}

\title[Hausdorff dimension of particularly non-normal
numbers]{Hausdorff dimension of particularly non-normal numbers in
  dynamical systems fulfilling the specification property}

\author[M. G. Madritsch]{Manfred G. Madritsch}
\address[M. G. Madritsch]{
  \noindent 1. Universit\'e de Lorraine, Institut Elie Cartan de
  Lorraine, UMR 7502, Vandoeuvre-l\`es-Nancy, F-54506, France;\newline
  \noindent 2. CNRS, Institut Elie Cartan de Lorraine, UMR 7502,
  Vandoeuvre-l\`es-Nancy, F-54506, France}
\email{manfred.madritsch@univ-lorraine.fr}

\author[I. Petrykiewicz]{Izabela Petrykiewicz}
\address[I. Petrykiewicz]{Max Planck Institute for Mathematics,
  Vivatsgasse 7, 53111 Bonn, Germany}
\email{petrykii@mpim-bonn.mpg.de}


\date{\today}

\begin{document}

\begin{abstract}
  In this paper, we consider non-normal numbers occurring in dynamical
  systems fulfilling the specification property. It has been shown
  that in this case the set of non-normal numbers has measure zero.
  In the present papers we show that a smaller set, namely the set of
  particularly non-normal numbers, has full Hausdorff dimension. A
  particularly non-normal number is a number $x$ such that there exist
  two digits, one whose limiting frequency in $x$ exists and another
  one whose limiting frequency in $x$ does not exist.
\end{abstract}

\maketitle

\section{Introduction}

Let $N\geq2$ be an integer and $\mathcal{A}:=\{0,1,\ldots,N-1\}$. Then for every
$x\in[0,1)$ we denote by
\begin{equation}\label{eq:Nadicexpansion}
   x=\sum_{k=1}^\infty d_k(x)N^{-k},
\end{equation}
where $d_k(x)\in\mathcal{A}$ for all $k\geq1$, the unique non-terminating
$N$-ary expansion of $x$. For every positive integer $n$ and a block of digits
$\mathbf{b}=b_1\ldots b_k\in\mathcal{A}^k$ we write
\[
  P_\mathbf{b}(x,n):=\frac{\lvert\{0\leq
    i<n:d_{i+1}(x)=b_1,\ldots,d_{i+k}(x)=b_k\}\rvert}n
\]
for the frequency of the block $\mathbf{b}$ among the first $n$ digits of the
$N$-ary expansion of $x$. 

We call a number $k$\textit{-normal} if for every block
$\mathbf{b}\in\mathcal{A}^k$ of digits of length $k$, the limit of the
frequency $\lim_{n\to\infty}P_\mathbf{b}(x,n)$ exists and equals
$N^{-k}$.  A number is called \textit{normal} with respect to base $N$
if it is $k$-normal for all $k\geq1$. Furthermore, a number is called
\textit{absolutely normal} if it is normal to any base $N\geq2$. On
the other hand, a number is called \textit{essentially non-normal} if
for all $i\in A$, the limit $\lim_{n\to\infty}P_i(x,n)$ does not
exist, and \textit{particularly non-normal} if there exists $i,j\in A$
such that $\lim_{n\to\infty}P_i(x,n)$ exists, but
$\lim_{n\to\infty}P_j(x,n)$ does not exist.

Already in 1909 Borel \cite{borel1909:les_probabilites_denombrables}
showed that the set of absolutely normal numbers has Lebesgue measure
1. Recently the fractal properties of different variants of non-normal
numbers have been of interest, see for example
\cite{albeverio_pratsiovytyi_torbin2005:topological_and_fractal,
  olsen2004:applications_multifractal_divergence02,
  olsen2004:applications_multifractal_divergence01,
  olsen_winter2003:normal_and_non}.  In particular, Albeverio,
Prats'ovytyi and Torbin
\cite{albeverio_pratsiovytyi_torbin2005:singular_probability_distributions}
showed in 2005 that the set of particularly non-normal numbers is
superfractal, meaning that its Lebesgue measure is 0, whereas its
Hausdorff dimension is 1.

Instead of considering $N$-ary expansions for an integer $N$ we might
take any $\beta>1$ and represent reals in $[0,1]$ in a similar
way. These representations are called $\beta$-expansions and were
introduced by Rényi \cite{renyi1957:representations_real_numbers}. For
$\beta>1$ let $S_\beta\colon[0,1]\to[0,1]$ be the transformation given
by
\begin{equation*}
S_\beta(x)=\beta x\mod 1
\end{equation*}
and $\mathcal{D}_\beta=\{0,\ldots,\lfloor\beta\rfloor \}$ be the
corresponding set of digits. Then every $x\in[0,1]$ admits a unique
representation of the form
\begin{equation*}
x=\sum_{n=1}^\infty \frac{d_n}{\beta^n},
\end{equation*}
with
$d_n=\lfloor \beta S^{n-1}_\beta x \rfloor \in \mathcal{D}_\beta$. For
$\beta>1$ and $x\in[0,1]$ we denote by
$\mathrm{d}_\beta(x)=d_1d_2d_3\ldots\in\mathcal{D}_\beta^\N$ the
$\beta$-expansion of $x$. We note that for $\beta=N\in\Z$ this yields
the $N$-adic expansion described above.


Let $\mathcal{L}_\beta$ denote the set of all $\beta$-expansions of
elements in $[0,1]$. A finite word (resp. a sequence) is called
$\beta$-admissible if it is a factor of an element (resp. an element)
of $\mathcal{L}_\beta$.

Every $x\in[0,1]$ has a $\beta$-expansion, however, not every sequence
$u\in\mathcal{D}_\beta^\N$ occurs in $\mathcal{L}_\beta$. In
connection with a characterization
Parry~\cite{parry1960:eta_expansions_real} observed that the expansion
of $1$ plays a crucial role. In particular, a sequence
$u\in\mathcal{D}_\beta^\N$ occurs as $\beta$-expansion if and only if
\begin{enumerate}
\item the $\beta$-expansion of $1$ is infinite, then
$S^k(u)\prec_{\textrm{lex}}d_\beta(1)$ for all $k\geq0$,
\item the $\beta$-expansion of $1$ is finite, $\textit{i.e.}$
  $d_\beta(1)=d_1d_2\ldots d_k0000\ldots$, then
  $S^k(u)\prec_{\textrm{lex}}\overline{d_1\ldots d_k}$ for all $k\geq0$,
\end{enumerate}
where $S$ denotes the shift and $\prec_{\textrm{lex}}$ denotes the
lexicographic ordering. According to this one calls $\beta>1$ a Parry
number if $d_\beta(1)$ is eventually periodic. In this case
$\beta$-expansions fulfills the so-called specification property which
we will defined in the following section.

The transformation $S_\beta$ is ergodic, and the invariant measure for
$S_\beta$ is given by
\begin{equation*}
\nu_\beta(B)=\int_B \Big(\frac{1}{F(\beta)} \sum_{\substack{n\geq 0\\ x< S^n_\beta 1}} \frac{1}{\beta^n} \Big)dx,
\end{equation*}
where
\begin{equation*}
F(\beta)=\int_0^1 \Big( \sum_{\substack{n\geq 0\\ x< S^n_\beta 1}}
\frac{1}{\beta^n} \Big)dx
\end{equation*}
is the normalising constant (\textit{cf.} Dajani and Kraaikamp
\cite{dajani_kraaikamp2002:ergodic_theory_numbers}). Using this
measure one can easily carry over the definitions of normal and
non-normal numbers to $\beta$-expansions. It has been shown
in~\cite{madritsch2014:non_normal_numbers} for full shifts and more
generally in \cite{madritsch_petrykiewicz2014:non_normal_numbers} for
systems fulfilling the specification property (under some mild technical
conditions) that the set of essentially non-normal numbers is residual
or comeagre (a countable intersection of everywhere dense sets).

In the present paper, we calculate the Hausdorff dimension of sets of
particularly non-normal numbers in dynamical systems fulfilling the
specification property. To this end we will introduce the necessary
notation for the statement of the results in the following
section. Then we construct a subset of the set of numbers, show that
they are all particularly non-normal numbers and calculate their
Hausdorff dimension in the last section.

\section{Notation and results}
Our notation follows Lind and Marcus
\cite{lind_marcus1995:introduction_to_symbolic}. Let $M$ be a compact
metric space and let $\varphi\colon M\to M$ be a continuous
mapping. Then $(M,\varphi)$ is called a topological dynamical
system.

For $x\in M$ we denote by $\delta_x$ the Dirac measure, which is
concentrated on $x$. The empirical measure (of order $n$) is defined
as \[T_n(x):=\frac1n\sum_{k=0}^{n-1}\delta_{\varphi^k(x)}.\]
Note that for a given subset $U\subset M$, $T_n(x)[U]$ is the
probability of $\varphi^k(x)$ being in $U$ for $0\leq k< n$,
\textit{i.e.}\[T_n(x)[U]=\frac1n\#\{0\leq k<n\colon \varphi^k(x)\in
U\}.\]
If $M$ is the torus $\mathbb{T}=\R/\Z$ and $U\subset\mathbb{T}$ is a
subinterval, then this corresponds to questions on uniform distribution
of the sequence $\{\varphi^k(x)\}_k$. We refer the interested reader
to Kuipers and
Niederreiter~\cite{kuipers_niederreiter1974:uniform_distribution_sequences},
Drmota and Tichy~\cite{drmota_tichy1997:sequences_discrepancies_and}
or Bugeaud~\cite{bugeaud2012:distribution_modulo_one} for more details
on uniform distribution.

In the present paper we consider normal numbers and therefore we need
a ``digital'' structure on the set $M$. To this end let
$\mathcal{P}=\{P_0,\ldots,P_{N-1}\}$ be a collection of disjoint open
subsets of $M$. Then we call $\mathcal{P}$ a topological partition of
$M$ if it is the union of the closures of the $P_i$, \textit{i.e.}
\[M=\overline{P_0}\cup\cdots\cup\overline{P_{N-1}}.\]
For the rest of the paper we fix a topological dynamical system
$(M,\varphi)$ together with a topological partition $\mathcal{P}$. We
note that fixing $\mathcal{P}$ and $\varphi$ suffices.

Now we construct the associated symbolic dynamical system. Let
$\mathcal{A}=\{0,1,\ldots,N-1\}$ be a finite set, the alphabet. We denote by
$\mathcal{A}^\N$ the set of infinite words equipped with the product
discrete topology. Furthermore let $S\colon\mathcal{A}^\N\to \mathcal{A}^\N$ be the shift
operator, \textit{i.e.} for $\omega=a_1a_2a_3\ldots\in \mathcal{A}^\N$
\[S(\omega)=a_2a_3a_4\ldots.\]

Let $\gamma\in\mathcal{A}^k$ be a finite word, then we write
$\lvert\gamma\rvert=k$ for its length. Furthermore for every
$\omega=a_1a_2a_3\ldots\in \mathcal{A}^\N$ and $n\geq1$ a positive integer, we denote
by $\omega\vert_n=a_1a_2\ldots a_n$ the truncation of $\omega$ to the
first $n$ letters. For a finite word $\omega=a_1a_2\ldots a_k\in A^k$
we denote by $[\omega]$ the corresponding cylinder set of order $k$,
\textit{i.e.} the set of all infinite words starting with the same
letters as $\omega$,
\[[\omega]=\{\gamma=b_1b_2b_3\ldots\in A^\N\colon a_i=b_i\text{ for
}1\leq i\leq k\}.\]

Let $\mathcal{A}^*=\bigcup_{k=1}^\infty\mathcal{A}^k$ be the set of all
finite words. A subset $\mathcal{L}\subset\mathcal{A}^*$ is called a
language. We call a word $\omega=a_1\ldots a_n$ allowed for the pair
$(\mathcal{P},\varphi)$ if the set
$\bigcap_{j=1}^n\varphi^{-j}P_{a_j}$ is not empty. Then we denote by
$\mathcal{L}:=\mathcal{L}_{\mathcal{P},\varphi}$ the set of all
allowed words for the pair $(\mathcal{P},\varphi)$. There is a unique
shift space $X:=X_{\mathcal{P},\varphi}$ whose language is
$\mathcal{L}_{\mathcal{P},\varphi}$ and we call $X$ the symbolic
dynamical system corresponding to $(\mathcal{P},\varphi)$.

We denote by $\mathcal{M}_S$ the set of all shift invariant
probability measures. A measure $\mu$ is called shift invariant if for
all measurable sets $A$ we have that
$\mu(A)=\mu(S^{-1}A)$. Furthermore we denote by $\mathfrak{B}$ the set
of Borel sets that is generated by the cylinder sets. Then for any
$\mu\in\mathcal{M}_S$ the tuple $(X,\mathfrak{B},S,\mu)$ describes a
measure theoretical dynamical system.

By abuse of notation we denote by $T_n$ also the empirical measure in $X$:
\[T_n(\omega):=\frac1n\sum_{k=0}^{n-1}\delta_{S^k(\omega)}.\]
Then for each block $\mathbf{b}\in \mathcal{A}^k$, $T_n(\omega)([\mathbf{b}])$ counts
the number of occurrences of the block $\mathbf{b}$ among the first
$n$ digits. This yields the identity
\[P_{\mathbf{b}}(\omega,n)=T_n(\omega)([\mathbf{b}]).\]
Now we can pull over the definition of particularly non-normal
numbers. We call an infinite word $\omega\in X$ particularly
non-normal if there exist two letters $i,j\in\mathcal{A}$ such that
\[\lim_{n\to\infty} T_n(\omega)([i])\text{ exists}\quad\text{and}\quad
\lim_{n\to\infty} T_n(\omega)([j])\text{ does not exist.}\]

We denote by $\mathcal{L}_n$ the subset of words of length at most
$n$, \textit{i.e.} $\mathcal{L}_n:=\{w\in \mathcal{A}^*\colon \lvert
w\rvert\leq n\}$.
Then the Shannon entropy of $\alpha\in\mathcal{M}_S$ is defined as
\[\entropy(\alpha):=\lim_{n\to\infty}-\frac1n \sum_{w\in\mathcal{L}_n}\alpha([w])\log(\alpha([w])).\]
A measure $\nu\in\mathcal{M}_S$ is called maximal (or an equilibrium
state) if it maximises this entropy. For a symbolic dynamical system
fulfilling the specification property there need not be a unique
maximal ergodic measure, however, the set of ergodic measures is a
dense $G_\delta$ subset of $\mathcal{M}$ (\textit{cf.} Proposition
21.9 of Denker \textit{et al.}
\cite{denker_grillenberger_sigmund1976:ergodic_theory_on} or Sigmund
\cite{sigmund1970:generic_properties_invariant}). For the case of
$\beta$-expansions Bertrand-Mathis
\cite{bertrand-mathis1988:points_generiques_de} showed that the
Champernowne word is generic for the unique maximal measure.

For showing that an element $\omega\in\mathcal{A}^\N$ is in fact
particularly non-normal we we look at the limit-points of the
frequency vectors $P_{\mathbf{b}}(\omega,n)$. On the other hand in
order to calculate the Hausdorff-dimension we consider measures of
cylinder sets. In particular, we need information on the limit-points
of the empirical measure of the constructed words.
Let $f\colon\mathcal{A}^\N\to\R$. Then we call $f$ local if there
exist indices $1\leq i\leq j<\infty$ such that the function values
coincide whenever the sub word on position $i$ up to $j$ coincides,
\textit{i.e.} $f(\omega)=f(\eta)$ whenever $\omega_k=\eta_k$ for
$i\leq k\leq j$. Let $\mathfrak{B}_n$ be the $\sigma$-algebra
generated by $\mathcal{L}_n$ and let $\mathcal{F}_n$ be the set of
local functions which are $\mathfrak{B}_n$-measurable.
Now our topology on $\mathcal{M}$, which coincides with the
weak\textsuperscript{*} topology, is given by the norm
\[\lVert\rho\rVert:=\sum_{n=1}^\infty2^{-n}\lVert\rho\rVert_{\mathrm{TVn}},
\quad\text{where}\quad
\lVert\rho\rVert_{\mathrm{TVn}}:=\sup_{f\in\mathcal{F}_n,\lVert
  f\rvert\leq1}\lvert\langle f,\rho\rangle\rvert.\]

For each $\omega\in\mathcal{A}^\N$ the sequence $\{T_n(\omega)\}_n$
has limit-points, which are clearly $S$-invariant probability
measures. If for an $\omega\in\mathcal{A}^\N$ the sequence
$\{T_n(\omega)\}_n$ has only one limit point $\alpha$, then we call
$\omega$ normal with respect to the measure $\alpha$ (or generic for
the measure $\alpha$). Furthermore we call $\omega\in\mathcal{A}^\N$
normal if it is normal with respect to the maximal measure.

Before we start with our construction, we present some hypothesis which
we suppose for the rest of this paper. All of them are fulfilled by
our motivating example -- the $\beta$-expansion.

\begin{itemize}
\item[\textbf{H1}] \textit{Different measures:} Let
  $\nu\in\mathcal{M}_S$ be the maximal measure and
  $\mu\in\mathcal{M}_S$ be any measure different from the maximal one
  such that there exists $i,j\in\{0,\ldots,N-1\}$ with
\[\mu([i])>\nu([i])\quad\text{and}\quad\mu([j])=\nu([j]).\]
Without loss of generality, we assume that $i=0$ and $j=1$.
\item[\textbf{H2}] \textit{Specification property:} 
There exists $C \in \N$ such that 
for any pair $a,b\in\mathcal{L}$ there exists $v\in\mathcal{L}$
with $\lvert v\rvert\leq C$ such that
$avb\in\mathcal{L}$. We will write $a\odot b:=avb$ for short.
\item[\textbf{H3}] \textit{Approximation property:} The maximum
  measure $\nu\in\mathcal{M}_S$ is non-atomic. There exists a
  continuous non-negative function $e_\nu$ on $\mathcal{L}$ such that
  \begin{gather}\label{ps:eq2.1}
  \limsup_n\sup_{\omega\in\mathcal{L}}\lvert\frac1n\log\nu[\omega\vert_n]+\langle
  e_\nu,T_n(\omega)\rangle\rvert=0
  \end{gather}
  and
  \[\exists C_\nu>0\quad\text{such that }\langle e_\nu,\rho\rangle\geq
  C_\nu\quad\forall\rho\in\mathcal{M}_S.\]
\end{itemize}

\textbf{Comments:}
\begin{itemize}
\item The hypothesis \textbf{H1} guarantees that we have words which
  tend to the different measures. It is clear from that, that there
  have to be at least 3 letters since otherwise $\nu([0])+\nu([1])=1$
  and if $P_0(\omega,n)$ has a limit, then also $P_1(\omega,n)$.
\item We note that we may have relaxed $\textbf{H2}$ by letting the
  length of the word $v$ depend on the length of the left and right
  word, respectively. However, this makes notation uglier and so we
  omitted it. For a different definition of the specification
  property, we refer the interested reader to Chapter 22 of
  \cite{denker_grillenberger_sigmund1976:ergodic_theory_on}.
\end{itemize}

\begin{thm}\label{thm:HausdorffDim}
  Let $(M,\varphi)$ be a topological dynamical system. Suppose that
  $\mathcal{P}$ is a finite topological partition of $M$ such that the
  associated symbolic dynamical system $(X,S)$ satisfies \textbf{H1},
  \textbf{H2} and \textbf{H3}. Then the set of particularly non-normal
  numbers has full Hausdorff dimension.
\end{thm}

As an example of a dynamical system satisfying \textbf{H1},
\textbf{H2} and \textbf{H3} we consider the $\beta$-expansions from
the introduction.

\begin{thm}
  Let $\beta>2$ be a Parry number. Let $(X,S)$ be the symbolic dynamical
  system corresponding to the $\beta$-expansion. Then the set of
  particularly non-normal numbers has full Hausdorff dimension.
\end{thm}

\begin{proof}
Let $\nu\in \mathcal{M}_{S}$ be the maximal measure. 
Without loss of generality, assume that $\nu([0])\neq 0$ and $\nu([1])\neq 0$. 
Define $\mu$ as follows. 
If 1 is a factor of $\omega \in \mathcal{L}$, then $\mu([\omega])=0$. 
Otherwise if 0 is not a factor of $\omega$, then $\mu([\omega])=\nu([\omega])$.
Finally, if 0 is a factor and 1 is not a factor of $\omega$, 
then considering the map $F:\mathcal{D}_\beta\to\mathcal{D}_\beta\setminus\{1\}$ 
which maps $1 \mapsto 0$ and $d\mapsto d$ for $d\neq 1$, we let
$\mu([\omega])=\sum_{u\in F^{-1}(\omega)\cap \mathcal{L}} \nu([u])$.
The constructed measure $\mu \in \mathcal{M}_S$ satisfies \textbf{H1}.
Moreover, if $\beta$ is a Parry number, then the system is sofic
by definition. By \cite[Section 5]{pfister_sullivan2003:billingsley_dimension_on},
the sofic shifts satisfy \textbf{H2} and \textbf{H3}.
The result then follows by Theorem~\ref{thm:HausdorffDim}.
\end{proof}

\section{The construction}

The aim of this section is to construct for every positive integer $p$
a set $G_p$ consisting of particularly non-normal numbers and having
Hausdorff dimension $\frac{p}{p+1}$. In particular, the elements of
the set $G_p$ will consist of two parts:
\begin{enumerate}
\item A fixed part constructed by concatenating parts of the sequences
  $\omega$ whose empirical measure tends to $\mu$ and
\item a changing part constructed by concatenating words whose empirical
  measure tends to $\nu$.
\end{enumerate}
We start by defining sets of finite words, whose empirical measure
tends to a given measure $\alpha$.

\begin{prop}[{\cite[Theorem
    2.1]{pfister_sullivan2003:billingsley_dimension_on}},
{\cite[Lemma 2.4]{lewis_pfister_russell+1997:reconstruction_sequences_and}}]
\label{ps:thm2.1}
Let $\alpha\in\mathcal{M}_S$. Then there exists a sequence $\{\Gamma(\alpha,n)\subset \mathcal{A}^n\colon n\in\N\}$ such that
\[\lim_{n\to\infty}\frac1n\log\lvert\Gamma(\alpha,n)\rvert=\entropy(\alpha),\]
and for all $\varepsilon>0$ there exists $N\in\N$ such that for all $n\geq N$ we have
\begin{equation}\label{eq:convtodistribution}
|P_j(\omega)-\alpha([j])|\leq \varepsilon
\end{equation}
for $j\in\{0,1\}$ and all $\omega\in \Gamma(\alpha,n)$.
\end{prop}

Corollary 21.15 of Denker \textit{et al.}
\cite{denker_grillenberger_sigmund1976:ergodic_theory_on} tells us
that every measure has a generic word. Therefore we choose $\omega \in X$
such that $P_i(\omega,n)\to\mu([i])$ as $n\to \infty$ for
$i\in\{0,1\}$ (the fixed part) and let $(\Gamma(\nu,n))_n$ be given as
in Proposition~\ref{ps:thm2.1} (the changing part). Since we need to
show that the different parts of the sequences in the set $G_p$ tend
to different measures, we have to keep track of their speed of
convergence. Thus on the one hand we define for each $n\in\N$ the
discrepancy $\delta_n=\max_{i\in\{0,1\}}|P_i(\omega,n)-\mu([i])|$. On
the other hand, for each $n\in \N$ we define the parameters $\eta_n$
and $\varepsilon_n$ by
$\left|\frac1n\log\lvert\Gamma(\alpha,n)\rvert-\entropy(\alpha)\right|=\eta_n$
and
$\varepsilon_n=\max_{i\in\{0,1\}}\max_{\omega\in\Gamma(\nu,n)}|P_i(\omega)-\nu([i])|$,
respectively. Finally we combine these distances by setting
$\varrho_n=\max\{\varepsilon_n,\delta_n\}$. Note that $\varrho_n\to 0$
as $n\to \infty$.

Let $B_0^{(1)}$ consist only of the empty word $\epsilon$. We set the length
of the initial block as $1$ and choose $n_1$ such that for $\omega_1=\omega\vert_{1}$ 
(the first letter of the word $\omega$) we have
\begin{equation*}
 |P_k(\omega_1^{\odot n_1})-\mu([k])|\leq 2\varrho_{1},
\end{equation*}
for $k\in\{0,1\}$. For $1\leq i \leq n_1$, let $B_i^{(1)}=\{\omega_1^{\odot i}\}$. 
Denote by $\ell^{(1)}_i$ the length of the word in the set $B_i^{(1)}=\{\omega_1^{\odot i}\}$.
Then for $n_1 +1 \leq i < (p+1)n_1$ let 
\begin{equation*}
 \widetilde{B}^{(1)}_i=B^{(1)}_{i-1}\odot\Gamma(\nu,1)
\end{equation*}
and we choose $B^{(1)}_{i}\subseteq \widetilde{B}^{(1)}_i$ of maximal
cardinality such that each element has the same length~$\ell^{(1)}_i$.

Now we continue recursively. To this end we suppose that the sets $B^{(j-1)}_{i}$
with $0\leq i\leq(p+1)n_{j-1}-1$ have already been constructed. We set
\begin{equation*}
  \widetilde{B}^{(j)}_0=B^{(j-1)}_{(p+1)n_{j-1}-1}\odot\Gamma(\nu,{j-1}),
\end{equation*}
and choose $B^{(j)}_{0}\subseteq \widetilde{B}^{(j)}_0$ of maximal
cardinality such that each element has the same length denoted by
$\ell^{(j)}_0$.  Let $\omega_j=\omega|_{j}$. Choose $n_j>n_{j-1}$ such
that
\begin{equation*}
  |P_i(\widetilde{\omega}\odot\omega_j^{\odot n_j})-\mu([i])|\leq 2\varrho_{j},
\end{equation*}
for $i\in\{0,1\}$ for all $\widetilde{\omega}\in B^{(j)}_{0}$, and 
\begin{equation}\label{eq:particularly1p3}
 \frac{\ell^{(j)}_0}{n_jj}\to 0 \textnormal{ as } j\to \infty.
\end{equation}
For $1 \leq i \leq n_j$, let 
\begin{equation*}
B^{(j)}_i=B^{(j)}_0\odot\omega_j^{\odot i}.
\end{equation*}
Finally, for $n_j +1 \leq i < (p+1)n_j$ let 
\begin{equation*}
 \widetilde{B}^{(j)}_i=B^{(1)}_{i-1}\odot\Gamma(\nu,j)
\end{equation*}
and again we choose $B^{(j)}_{i}\subseteq \widetilde{B}^{(j)}_i$ of
maximal cardinality such that each element has the same length
$\ell^{(j)}_i$.

We may now define our set $G_p$ as the limit of the sets $B^{(j)}_i$,
\textit{i.e.}
\[G_p=\bigcap_{j\geq1}\bigcap_{i}\bigcup_{b\in B^{(j)}_i}[b].\]
By the construction we obtain that each element in $G_p$ has a unique
representation of the form
\begin{align*}
&\underbrace{\omega_1\odot\omega_1\odot\cdots\odot\omega_1}_{n_1
  \textnormal{ times}}\odot a_1^{(1)}\odot a_2^{(1)}\odot\cdots\odot
  a_{pn_1}^{(1)}\odot \\
&\underbrace{\omega_2\odot\omega_2\odot\cdots\odot\omega_2}_{n_2 \textnormal{
  times}}\odot a_1^{(2)}\odot a_2^{(2)}\odot\cdots\odot
  a_{pn_2}^{(2)}\odot\\
&\vdots\\
&\underbrace{\omega_j\odot\omega_j\odot\cdots\odot\omega_j}_{n_j \textnormal{ times}}\odot a_1^{(j)}\odot a_2^{(j)}\odot\cdots\odot a_{pn_j}^{(j)}\odot\cdots
\end{align*}
with $\omega_j=\omega\vert_j$ and $a_{k}^{(j)}\in \Gamma(\nu,{j})$ for
$1\leq k \leq pn_j$.

We finish this section by providing a lower bound for the number of elements in $B_i^{(j)}$.
\begin{lem}\label{lem:num_elem_estimate}
 We have the following:
\begin{align*}
 |B^{(j)}_i|&=|B^{(j)}_0| \textnormal{ for all } j \textnormal{ and } 1\leq i \leq n_j;\\
 |B^{(j)}_i|&\geq \exp\left(p\sum_{k=1}^{j-1}n_kk\left(\entropy-\eta_{k}-\frac{\log C}{k}\right)+(i-n_j)j\left(\entropy-\eta_{j}-\frac{\log C}{j}\right)\right)\\
& \quad \quad \quad \textnormal{ for all } j \textnormal{ and } n_j+1\leq i < pn_j;\\
 |B^{(j)}_0|&\geq \exp\left(p\sum_{k=1}^{j-1}n_kk\left(\entropy-\eta_{k}-\frac{\log C}{k}\right)\right) \textnormal{ for all } j\geq 1.
\end{align*}
\end{lem}
\begin{proof}
 It follows by induction from our construction and Proposition~\ref{ps:thm2.1}.
\end{proof}

\section{Particularly non-normal}

In this section we want to show that each element of $G_p$ is
particularly non-normal, \textit{i.e.} that there exist two digits
such that there is no limit frequency for the first one but certainly
there is one for the second one. Let $N_i(\omega)$ denote the number
of occurrences of a digit $i$ in a word $\omega$, and if $\omega$ is
finite write $P_i(\omega)$ to denote $P_i(\omega, |\omega|)$ the
frequency of occurrences of the digit $i$ in the word $\omega$.

We start by showing that for all $\omega\in G_p$ we have that
$P_1(\omega, m)\to \nu([1])=\mu([1])$ as $m \to \infty$.
Let $\omega \in G_p$ and $m\in \N$. Choose $j$ such that $\ell_0^{(j)}<m\leq \ell_0^{(j+1)}$.
First assume that $\ell_{n_j}^{(j)}<m$. In this case,
\begin{equation*}
\omega|_m=\underbrace{\omega_1\odot\cdots\odot\omega_j}_{\omega|_{\ell_{n_j}^{(j)}}}\odot a_1^{(j)}\odot\cdots\odot a_{k-1}^{(j)}\odot \widetilde{a}_{k}^{(j)},
\end{equation*}
with $0\leq k \leq pn_j$. Let $v_i$ be the connecting block inserted
before the word $a_i^{(j)}$.  Then we can write 
$m=\ell_{n_j}^{(j)}+(k-1)j+\sum_{i=1}^k|v_i|+|\widetilde{a}_{k}^{(j)}|$
and get that
\begin{equation*}
 P_1(\omega,m)=P_1(\omega,\ell_{n_j}^{(j)})\frac{\ell_{n_j}^{(j)}}{m}+\frac{j}{m}\sum_{i=1}^kP_1(a_i^{(j)},j)
+\frac{1}{m}\sum_{i=1}^k N_1(v_i)+\frac{N_1(\widetilde{a}_{k}^{(j)})}{m}.
\end{equation*}
Note that
\begin{align*}
 0\leq & \frac{1}{m}\sum_{i=1}^k N_1(v_i) 
\leq \frac{\sum_{i=1}^k|v_i|}{\ell_{n_j}^{(j)}+(k-1)j+\sum_{i=1}^k|v_i|+|\widetilde{a}_{k}^{(j)}|}
\leq \frac{kC}{(n_j+k-1)j}\xrightarrow[m\to\infty]{} 0\\
\intertext{and}
 0 \leq & \frac{N_1(\widetilde{a}_{k}^{(j)})}{m}  
\leq \frac{|\widetilde{a}_{k}^{(j)}|}{\ell_{n_j}^{(j)}+(k-1)j+\sum_{i=1}^k|v_i|+|\widetilde{a}_{k}^{(j)}|}
\leq \frac{j}{\ell_{n_j}^{(j)}} \leq \frac{1}{n_j}\xrightarrow[m\to\infty]{} 0.
\end{align*}
Finally, by the construction, we have
\begin{equation*}
 |P_1(\omega,\ell_{n_j}^{(j)})-\nu([1])|\leq 2\varrho_{j}, 
\end{equation*}
and for all $i$
\begin{equation*}
 |P_1(a_i^{(j)},j)-\nu([1])|\leq \varrho_{j}.
\end{equation*}
Therefore, we conclude that 
\begin{equation}\label{eq:particularly1p1}
 P_1(\omega,m) \to \nu([1]) \textnormal{ as } m\to \infty, \ell_{n_j}^{(j)}<m.
\end{equation}

Suppose now $m\leq\ell_{n_j}^{(j)}$. In this case, there exists a $k$
with $0\leq k<n_j$ such that
\begin{equation*}
\omega|_m=\underbrace{\omega_1\odot\cdots\odot \omega_{j-1}\odot\cdots\odot a_{pn_{j-1}}^{(j-1)}}_{\omega|_{\ell_{0}^{(j)}}}\odot
\underbrace{\omega_j\odot\cdots\odot \omega_j}_{\omega_j \textnormal{ repeated } k \textnormal{ times}}\odot\;\widetilde{\omega}_j.
\end{equation*}
Let $v_i$ be the connecting block inserted before the $i$th $\omega_j$ word. 
We have $m=\ell_{0}^{(j)}+kj+\sum_{i=1}^{k+1}|v_i|+|\widetilde{\omega}_j|$.
With this notation, we have
\begin{equation*}
 P_1(\omega,m)=P_1(\omega,\ell_{n_{j-1}}^{(j-1)})\frac{\ell_{n_{j-1}}^{(j-1)}}{m}+\frac{j}{m}\sum_{i=1}^{pn_{j-1}}P_1(a_i^{(j-1)},{j-1})
+\frac{kj}{m}P_1(\omega_j,j)+\frac{1}{m}\sum_{i=1}^{k+1} N_1(v_i)+\frac{N_1(\widetilde{\omega}_j)}{m}.
\end{equation*}
As before, we have
\begin{gather*}
\frac{1}{m}\sum_{i=1}^{k+1} N_1(v_i) \xrightarrow[m\to\infty]{} 0\quad\text{and}\quad
\frac{N_1(\widetilde{\omega}_j)}{m}\xrightarrow[m\to\infty]{} 0.
\end{gather*}
Again by construction
\begin{align*}
 |P_1(\omega,\ell_{n_{j-1}}^{(j-1)})-\nu([1])| 	& \leq 2\varrho_{{j-1}};\\
 |P_1(a_i^{(j-1)},{j-1}) - \nu([1])| 		& \leq \varrho_{{j-1}}\quad\text{for all } i;\\
 | P_1(\omega_j, j)-\nu([1])| 		& \leq \varrho_{j}.
\end{align*}
Therefore, we conclude that 
\begin{equation}\label{eq:particularly1p2}
 P_1(\omega,m) \to \nu([1]) \textnormal{ as } m\to \infty, m\leq \ell_{n_j}^{(j)}.
\end{equation}

From \eqref{eq:particularly1p1} and \eqref{eq:particularly1p2} we
conclude that $ P_1(\omega,m) \to \nu([1]) $ as $m \to \infty$ and
therefore the limiting frequency of the digit $1$ exists.

Now we show that the limiting frequency of digit $0$ does not
exist. On the one hand, we note that by our construction we have that
\begin{equation}\label{eq:particularly1p4}
P_0(\omega, \ell_{n_j}^{(j)}) \to \mu([0]) \textnormal{ as } j\to\infty.
\end{equation}
On the other hand, let $v_i$ be the connecting block inserted before
the word $a_i^{(j)}$. Then we have
\begin{equation*}
P_0(\omega, \ell_{0}^{(j+1)})=P_0(\omega,\ell_{n_j}^{(j)})\frac{\ell_{n_j}^{(j)}}{\ell_0^{(j+1)}}
+\frac{j}{\ell_0^{(j+1)}}\sum_{i=1}^{pn_j}P_0(a_{i}^{(j)},j)+\frac{1}{\ell_0^{(j+1)}}\sum_{i=1}^{pn_j} N_1(v_i).
\end{equation*}
As above we get that
\begin{align*}
\frac{1}{\ell_0^{(j+1)}}\sum_{i=1}^{pn_j} N_1(v_i)\to 0,\quad
P_0(\omega,\ell_{n_j}^{(j)})\to \mu([0])\quad\text{and}\quad
  P_0(a_{i}^{(j)},j)\to \nu([0])
\end{align*}
as $j\to \infty$. 
Finally, by \eqref{eq:particularly1p3}, we have
\[\frac{\ell_{n_j}^{(j)}}{\ell_0^{(j+1)}}\to \frac{1}{p}\quad\text{and}\quad\frac{pn_jj}{\ell_0^{(j+1)}}\to \frac{p}{p+1}.\] 
We conclude that
\begin{equation}\label{eq:particularly1p5}
P_0(\omega, \ell_{0}^{(j+1)}) \to
\frac{1}{p+1}\mu([0])+\frac{p}{p+1}\nu([0])<\mu([0])
\end{equation}
as $j\to\infty$. It follows from \eqref{eq:particularly1p4} and
\eqref{eq:particularly1p5} that the limiting frequency of the digit
$0$ in $\omega$ does not exist, proving that $\omega$ is indeed a
particularly non-normal number.

\section{Hausdorff dimension}

For the estimation of the Hausdorff dimension it suffices to consider
the cylinder sets.  The set $B_i^{(j)}$ belongs to the $i$th word of
the $j$th block. In the present section we want to ease notation by
saying the $n$th word is the $i$th word of the $j$th block. In
particular, for each $n$ there exist unique $i=i(n)$ and $j=j(n)$ such
that
\[n=(p+1)\sum_{k=1}^{j-1}n_k + i\]
with $1\leq i\leq (p+1)n_j$.  Similarly we set $B_n=B_i^{(j)}$,
\[y_n=y_i^{(j)}:=\min\{\nu([x])\colon x\in B_i^{(j)}\}\]
and
\[E_n=E_i^{(j)}:=\exp\left(p\sum_{k=1}^{j-1}n_kk\left(\entropy-\eta_{k}-\frac{\log
    C}{k}\right)+\sum_{k=n_b+1}^{i}j\left(\entropy-\eta_{j}-\frac{\log
    C}{j}\right)\right).\]

Then our first tool is the following lemma.
\begin{lem}\label{lem:cover_lower_bound}
Let $0<s\leq 1$ and let $\mathcal{D}$ be a cover of $G_p$ such that $\nu([w])<y_N$
for all $w\in\mathcal{D}$. Then there exists $n\geq N$ such that
\[
\sum_{w \in \mathcal{D}} \nu([w])^s
\geq E_ny_{n+1}^s.
\]
\end{lem}

\begin{proof}
Since $\nu([w])<y_N$ for all $w\in\mathcal{D}$, each $w\in\mathcal{D}$ has a prefix
in $B_N$, \textit{i.e.} $w=vu$ with $v\in B_N$. Moreover, since $\mathcal{D}$ is a cover, every $v\in B_N$ has to appear as prefix. Thus we may write
\[\sum_{w\in\mathcal{D}}\nu([w])^s=\sum_{v\in B_N}\sum_{vu\in\mathcal{D}}\nu([vu])^s.\]
Now let $\tilde{v}\in B_N$ be such that for all $v\in B_N$
\[\sum_{u\colon vu\in\mathcal{D}}\nu([vu])^s
  \geq\sum_{u\colon \tilde{v}u\in\mathcal{D}}\nu([\tilde{v}u])^s.\]
We have
\[\sum_{w\in\mathcal{D}}\nu([w])^s
  \geq\lvert B_N\rvert\sum_{u\colon\tilde{v}u\in\mathcal{D}}\nu([\tilde{v}u])^s.\]
Now, either
\[\sum_{u\colon \tilde{v}u}\nu([\tilde{v}u])^s\geq
\left(y_{N+1}\right)^s,\quad\text{or}\quad 
\sum_{u\colon \tilde{v}u}\nu([\tilde{v}u])^s<\left(y_{N+1}\right)^s.\]
If the first case holds true, then an application of Lemma
\ref{lem:num_elem_estimate} yields
\[\sum_{w\in \mathcal{D}}\nu([w])^s
  \geq \exp\left(p\sum_{k=1}^{j-1}n_kk\left(\entropy-\eta_{k}-\frac{\log C}{k}\right)+\sum_{k=n_j+1}^ij\left(\entropy-\eta_{j}-\frac{\log C}{j}\right)\right)\\
  y_{N+1}^s.\]
On the other hand, if the latter case is true, then let
$P_{N+1}:=\{v\in B_{N+1}\colon \tilde{v}\text{ is a prefix
  of }v\}$. Now every $u$ with $\tilde{v}u\in\mathcal{D}$ can be
written as $\tilde{v}u=v'u'$, where $v'\in P_{N+1}$. Since
$\mathcal{D}$ is a cover of $G_p$, all prefixes $v'\in P_{N+1}$
occur in the decomposition of $\tilde{v}u=v'u'$. Thus by construction
of $B_{N+1}$ we either have that
\[\lvert
P_{N+1}\rvert\geq\exp(p{j(N+1)}\left(\entropy-\eta_{{j(N+1)}})\right)\]
if $i(N+1)>n_{j(N+1)}$ or $\lvert P_{N+1}\rvert =1$ otherwise.
Similar to above we may choose $\tilde{v}'\in P_{N+1}$ such that
for all $v'\in P_{N+1}$,
\[\sum_{u'\colon v'u'\in\mathcal{D}}\nu([v'u'])^s
  \geq \sum_{u'\colon \tilde{v}'u'\in\mathcal{D}}
  \nu([\tilde{v}'u'])^s.\]

Thus combining this with the estimate above yields
\begin{align*}
\sum_{w\in\mathcal{D}}\nu([w])^s
&\geq E_n \sum_{u\colon \tilde{v}u\in\mathcal{D}}\nu([\tilde{v} u])^s\\
&=E_n \sum_{v'\in P_{N+1}} \sum_{u'\colon v'u'\in\mathcal{D}}\nu([v' u'])^s\\
&\geq E_{n+1} \sum_{u'\colon \tilde{v}'u'\in\mathcal{D}}\nu([\tilde{v}' u'])^s.
\end{align*}

Iterating this argument yields (after finitely many steps since
$\mathcal{D}$ is finite) that
\[
\sum_{w\in\mathcal{D}}\nu([w])^s
\geq E_n y_{n+1}^s.\qedhere\]
\end{proof}

The second ingredient is the following lemma of Pfister and Sullivan
\cite{pfister_sullivan2003:billingsley_dimension_on}.
\begin{lem}[{\cite[Lemma 3.1 and Corollary 3.1]{pfister_sullivan2003:billingsley_dimension_on}}]\label{ps:lem3.1}
Let $\nu$ be a probability measure which possesses a continuous
strictly positive regular conditional probability
$\nu(\omega_1|\omega_2,\omega_3,\ldots)$. Then,
\[e_\nu(\omega):=-\log\nu(\omega_1|\omega_2,\omega_3,\ldots)\]
satisfies \eqref{ps:eq2.1}. If $\nu$ satisfies \eqref{ps:eq2.1}, then,
for each $\delta>0$, there exist $m_\delta$ and $N_\delta\in\N$ such
that for all $n\geq N_\delta$ and for all $\omega\in\Sigma^\nu$
\[\lvert\langle e_\nu,T_n(\omega)\rangle+\frac1n\log\nu([\omega_1^n])\rvert<\delta.\]
Furthermore, we have
\[\entropy(\nu)
=\langle e_\nu,\nu\rangle.\]
\end{lem}

Now we are able to calculate the Hausdorff dimension. 

\noindent \textit{Proof of Theorem \ref{thm:HausdorffDim}}.
We suppose that $s<\frac{p}{p+1}$. Then by our construction
\begin{align*}
\ell_nT_{\ell_n}(\omega)
&=\sum_{k=1}^{j-1}n_kkT_{k}(\omega_k)+\sum_{k=1}^{j-1}\sum_{m=1}^{pn_k}kT_{k}\left(
  a^{(k)}_m\right)\\
&\quad+i j T_{j}(\omega_j)+
  \sum_{m=1}^{i-n_j}jT_{j}\left( a^{(j)}_m\right) +\mathcal{O}\left(\sum_{k=1}^{j-1}n_k(p+1)+i\right).
\end{align*}

Similar to the proof of $\omega$ being particularly non-normal we need
to distinguish two cases according to whether $0\leq i< n_j$ or
$n_j\leq i<n_j(p+1)$. For the first case we have
\begin{multline*}
\lVert T_{\ell_n}(\omega)
-\sum_{k=1}^j\frac{n_kk}{\ell_n}T_{k}(\omega_k)
-\frac{ij}{\ell_n}T_{j}(\omega_j)
-\sum_{k=1}^{j-1}\frac{pn_kk}{\ell_n}\nu\rVert\\
\leq
\frac1{\ell_n}\sum_{k=1}^{j-1}\sum_{m=1}^{pn_k}k\lVert T_{k}(a_m^{(k)})-\nu\rVert
\leq
\sum_{k=1}^{j-1}\frac{pn_kk}{\ell_n}\rho_{k}
\leq \frac{p}{p+1}\sum_{k=1}^j\rho_{k}.
\end{multline*}

Using the lower bound of Lemma \ref{lem:cover_lower_bound} together
with an application of Lemma \ref{ps:lem3.1} yields that for each $\delta>0$
there exists $N_\delta$ such that for $n\geq N_\delta$ we have
\begin{multline*}
p\sum_{k=1}^{j-1}n_kk\left(\entropy(\nu)-\eta_{k}\right)+s\log y_{n+1}\\
\geq\sum_{k=1}^{j-1}n_kk\entropy(\nu)\left(p-(p+1)s\right)
  -p\sum_{k=1}^{j-1}n_kk\eta_{k}
  -s(i+1)j\langle e_\nu,T_{j}(\omega_j)\rangle
  -s\delta\ell_{n+1}.
\end{multline*}

For the second case, $n_j\leq i<n_j(p+1)$, we get by similar means that
\[
\lVert T_{\ell_n}(\omega)
-\sum_{k=1}^j\frac{n_kk}{\ell_n}T_{k}(\omega_k)
-\sum_{k=1}^{j-1}\frac{pn_kk}{\ell_n}\nu
-\frac{(i-n_j)j}{\ell_n}\nu\rVert
\leq p\sum_{k=1}^{j-1}\frac{n_kk}{\ell_n}\rho_{k}+\frac{(i-n_j)j}{\ell_n}\eta_{j}.
\]

Again using the lower bound of Lemma \ref{lem:cover_lower_bound}
together with Lemma \ref{ps:lem3.1} we get that for each $\delta>0$
there exists $N_\delta$ such that for $n\geq N_\delta$ we have
\begin{multline*}
p\sum_{k=1}^{j-1}n_kk\left(\entropy(\nu)-\eta_{k}\right)+s\log y_{n+1}\\
\geq\sum_{k=1}^{j-1}n_kk\entropy(\nu)\left(p-(p+1)s\right)
  -p\sum_{k=1}^{j-1}n_kk\eta_{k}
  -sn_jj\langle e_\nu,T_{j}(\omega_j)\rangle
  -s(i+1-n_j)j\entropy(\nu)
  -s\delta\ell_{n+1}.
\end{multline*}

Thus in both cases by taking $\delta>0$ sufficiently small and noting that
$\sum_{k=1}^{j-1}n_kk/\ell_n=\frac{p}{p+1}$ we get that
\[\liminf_n E_ny_{n+1}^s=\infty,\]
provided that $s<\frac{p}{p+1}$. Thus
$\dim_H(G_p) \geq \frac{p}{p+1}$.  Letting $p \to \infty$ shows that
the Hausdorff dimension of the set of particularly non-normal numbers
is greater or equal to 1.  Since it cannot exceed 1, it completes the
proof of Theorem~\ref{thm:HausdorffDim}. \hfill \qedsymbol

\section*{Acknowledgment}

For the realization of the present paper the first author received
support from the Conseil R\'egional de Lorraine. Parts of this
research work were done when the authors were mutually visiting the
Institut Élie Cartan de Lorrain of the University of Lorraine and the
Max Planck Institute for Mathematics. The authors thank the
institutions for their hospitality.


\end{document}